\documentclass[psamsfonts,a4paper,11pt]{amsart}
\usepackage[all,arc]{xy}
\usepackage{tikz}
\usepackage[all]{xy}
\usepackage{latexsym}
\usepackage{amssymb}
\usepackage{multirow}
\usepackage{amsmath}
\newtheorem{definition}{Definition}[section]
\newtheorem{lemma}[definition]{Lemma}
\newtheorem{prop}[definition]{Proposition}
\newtheorem{theorem}[definition]{Theorem}

\newtheorem{conj}[definition]{Conjecture}

\newtheorem{ques}[definition]{Question}

\theoremstyle{definition}\newtheorem{exam}[definition]{Example}
\newtheorem{fact}[definition]{Fact}

\newcommand*{\ov}{\overline}

\newcommand*{\sq}{\sqrt}
\newcommand*{\bsm}{\left(\begin{smallmatrix}}
\newcommand*{\esm}{\end{smallmatrix}\right)}
\newcommand*{\bp}{\begin{pmatrix}}
\newcommand*{\ep}{\end{pmatrix}}

\newcommand*{\lf}{\lfloor}
\newcommand*{\rf}{\rfloor}

\renewcommand{\emph}{\textbf}

\newcommand{\End}{\mathop{{\rm End}}\nolimits}
\newcommand{\GL}{\mathop{{\rm GL}}\nolimits}

\newcommand{\GO}{\mathop{{\rm GO}}\nolimits}
\newcommand{\GU}{\mathop{{\rm GU}}\nolimits}
\newcommand{\Jac}{\mathop{{\rm Jac}}\nolimits}
\renewcommand{\O}{\mathop{\raisebox{0pt}{\rm O}}\nolimits}

\newcommand{\SL}{\mathop{{\rm SL}}\nolimits}
\newcommand{\SU}{\mathop{{\rm SU}}\nolimits}
\newcommand{\PSL}{\mathop{{\rm PSL}}\nolimits}
\newcommand{\PSU}{\mathop{{\rm PSU}}\nolimits}
\newcommand{\Soc}{\mathop{{\rm Soc}}\nolimits}
\newcommand{\Sp}{\mathop{{\rm Sp}}\nolimits}

\newcommand{\PSp}{\mathop{{\rm PSp}}\nolimits}
\newcommand{\Sz}{\mathop{{\rm Sz}}\nolimits}

\newcommand*{\F}{\mathbb{F}}

\newcommand*{\Q}{\mathbb{Q}}

\newcommand*{\Z}{\mathbb{Z}}

\newcommand*{\al}{\alpha}

\newcommand*{\eps}{\varepsilon}
\newcommand*{\ga}{\gamma}

\newcommand*{\lam}{\lambda}

\renewcommand*{\phi}{\varphi}

\begin{document}

\footskip=30pt

\title{When the group ring of a finite group over a field is serial}

\author{Andrei Kukharev}
\address{Faculty of Mathematics, Vitebsk State University, Moscow av. 33, Vitebsk 210038, Belarus}
\email{kukharev.av@mail.ru}

\author{Gena Puninski}
\address{Faculty of Mechanics and Mathematics, Belarusian State University, av. Nezalezhnosti 4,
Minsk 220030, Belarus}
\email{punins@mail.ru}

\thanks{The first named author was partially supported by BFFR grant F15RM-025.}


\begin{abstract}
We will describe finite simple groups whose group rings over a given field are serial.
\end{abstract}

\maketitle

\pagestyle{plain}

\section{Introduction}\label{S-intro}

Let $F$ be a field and let $G$ be a finite group. In this paper we will attempt to make a list of pairs $(F,G)$
such that the group ring $FG$ is serial, i.e.\ each indecomposable projective right (equivalently left) $FG$-module
has a unique composition series. Despite many important results has been proven, a complete description is
still elusive. First of all, if the characteristic $p$ of $F$ does not divide the order of $G$, then, by
Maschke's theorem, $FG$ is a semisimple artinian ring, hence serial. Thus we may assume that $p$ divides the
order of $G$, the modular case.

Since each artinian serial ring is of finite representation type, it follows from Higman \cite{Hig} that
seriality of $FG$ implies that each Sylow $p$-subgroup of $G$ is cyclic. If $p=2$ then this condition is
sufficient, because $G$ is 2-nilpotent in this case. This result can be extended to  $p$-solvable groups.
Namely, using Morita \cite{Mor}, one concludes that, if $G$ is a $p$-solvable group with a cyclic Sylow
$p$-subgroup, then the ring $FG$ is serial. Despite Morita worked over an algebraically closed field, this is
not a restriction, because Eisenbud and Griffith \cite{E-G} showed that seriality of $FG$ depends only on
characteristic of $F$.

However, if $G= \SL_2(5)$, then each Sylow $5$-subgroup of $G$ is cyclic, but the ring $F G$ is not serial for
any field of characteristic 5. Nevertheless, every group $G$ with a cyclic Sylow $p$-subgroup is close to a
$p$-solvable group. Namely, by Blau \cite{Blau}, if $G$ is not $p$-solvable, it admits a normal series
$\{e\}\subset O_{p'}\subset K\subset G$, where $O_{p'}$ is the largest normal subgroups of $G$ do not containing
elements of order $p$, and $K$ is the least normal subgroup of $G$ properly containing $O_{p'}$, in particular
$P\subset K$ and $H= K/O_{p'}$ is a simple nonabelian group. Here is the main conjecture.

\begin{conj}\label{c-main}
Suppose that $F$ is a field of characteristic $p$ dividing the order of $G$. Further assume that each Sylow
$p$-subgroup of $G$ is cyclic. Then the group ring $FG$ is serial if and only if either $G$ is $p$-solvable, or
the ring $FH$ is serial.
\end{conj}

In this paper we will complete the description of finite simple groups with serial group rings.

\begin{theorem}\label{t-main}
Let $H$ be a finite simple group and let $F$ be a field of characteristic $p$ dividing the order of $H$.
Then the group ring $FH$ is serial if and only if one of the following holds.

1) $H= C_p$.

2) $H= \PSL_2(q)$ and $p> 2$ divides $q-1$.

3) $H= \PSL_2(q)$, $q\neq 2$ or $H= \PSL_3(q)$, where $p=3$ and $q\equiv 2, 5 \pmod 9$.

4) $H= \PSU_3(q^2)$ and $p>2$ divides $q-1$.

5) $H= \Sz(q)$, $q= 2^{2n+1}$, $n\geq 1$, where either $p> 2$ divides $q-1$, or $p=5$ divides $q+r+1$, $r= 2^{n+1}$,
but $25$ does not divide this number.

6) $H= {}^2G_2(q^2)$, $q^2= 3^{2n+1}$, $n\geq 1$, where either $p>2$ divides $q^2-1$, or $p=7$ divides
$q^2+\sq{3} q+1$, but $49$ does not divide this number.

7) $H= M_{11}$, $p=5$ or $G= J_1$, $p=3$.
\end{theorem}

For instance, $A_5$ with $p=3$ occurs twice in this list: first as $\PSL_2(4)$ in 2), and then as $\PSL_2(5)$ in 3).

Using this theorem we verify Conjecture \ref{c-main} for groups of order $\leq 10^4$. However, we do not know
whether it holds true in general even for $p=3$; or if seriality of $FG$ implies seriality of $FH$ for each
normal subgroup $H$ of $G$.

The main tool in proving Theorem \ref{t-main} is the following well known characterization of seriality.
Namely, if $F$ is sufficiently large, then the group ring $FG$ is serial iff the Brauer tree of each $p$-block
of $G$ is a star whose exceptional vertex (if any) is located at its center. This allows us to use a well
developed machinery of Brauer trees to complete the proof of this theorem. In fact, due to previously obtained
results, the only remaining case is when $H$ is a classical (in fact symplectic, unitary or orthogonal) group
defined over a finite field of characteristic 2, and we address this case in the paper. We will try to make the
paper self-contained by subsuming and elaborating existing knowledge on serial group rings.

Of course, there is a broader context for the problem we consider. For instance, Tuganbaev
\cite[part of Probl. 16.9]{Tug} asks to characterize rings $R$ and (finite or infinite) groups $G$ such that
the ring $RG$ is serial. We will not comment on this general question, because it deserves a separate discussion,
but some connections are obvious. Say, if $R$ has a field as a factor, then all results of this paper are
applicable.

The essential part of this research is included in the first named author PhD thesis \cite{Kukh}, and otherwise
is scattered in \cite{K-P13a, K-P13b, VKP, K-P14, K-P15a, K-P15b, K-P16, K-P17}. The second named author
participated in this thesis as a supervisor. He thanks Chris Gill for the initial discussion in Prague in
Summer 2012, and for showing him how to check seriality using MAGMA. We are also indebted to Alexandre Zalesski,
Alexej Kondratiev, David Craven and Meinholf Geck for their comments.

\section{Basics}

All rings in this paper will be associative with unity, and, by default, a module means a unital right module
over a ring. A module $M$ is said to be \emph{uniserial}, if its lattice of submodules is a chain; and
$M$ is \emph{serial} if it is a direct sum of uniserial modules. We say that a ring $R$ is \emph{serial}, if
the right regular module $R_R$ is serial and the same holds true for the left module $_RR$. In fact, $R$ is
serial iff there exists a set $e_1, \dots, e_n$ of pairwise orthogonal idempotents which is \emph{complete},
i.e. $e_1+ \dots + e_n=1$, and each \emph{principal projective} module $e_iR$ is uniserial, so as each left
module $Re_i$. Further, this collection of idempotents is unique up to conjugation by a unit. For more on
general theory of serial rings the reader is referred to \cite{Punb}, or to more recent \cite{B-O}. We will
downsize to a more restrictive setting.

\emph{Artinian serial rings} were introduced by Nakayama (under the name of generalized uniserial rings), and
nowadays are known (see \cite{B-O}) as \emph{Nakayama rings}, or \emph{Nakayama algebras}, if they are algebras
over a field; though \cite{ARS} gives a broader meaning to the latter term. Recall that a serial ring is said
to be \emph{basic}, if different principal projective modules are not isomorphic. When classifying serial rings
one can always assume (modulo Morita equivalence) that $R$ is indecomposable and basic.

For instance, if $R$ is an indecomposable basic artinian serial rings with simple modules $S_i$, then
(see \cite[Cor. 12.4.2]{HGK}) all skew fields $D_i= \End (S_i)$ are isomorphic. Suppose that $R$ is  an
indecomposable basic finite dimensional Nakayama algebra over a perfect field and $D$ is the common value of
endomorphism rings of simple $R$-modules. It follows from Kupisch \cite[Satz 1, Hilfsatz 2.1]{Kup} that there
exists an automorphism $\al$ of $D$ and a proper (hence uniserial) factor $S= D[x, \al]/(x^k)$, with
Jacobson radical $J$, of the skew polynomial ring $D[x, \al]$ such that $R$ is a factor-ring of the
following \emph{blow-up} of $S$:

$$
\bsm S&S&\dots&S\\J&S&\dots&S\\\vdots&\vdots&\ddots&\vdots\\J&J&\dots&S\esm\,.
$$

\vspace{2mm}

Furthermore, these factor rings are easily described, - see \cite{Kup} for a complete list of invariants.
Using the blow-up construction a classification of artinian serial rings can be reduced to the case of
indecomposable basic rings with $n$ simple modules, such that the length of all principal projectives is a
constant equivalent to 1 modulo $n$; for instance, this happens for serial group rings of finite groups. No
reasonable classification is known in this case, - see \cite[p. 275--276]{B-O} for a list of open questions.

The very important property discovered by Nakayama is that any artinian serial ring $R$ is of
\emph{finite representation type}: each $R$-module is a direct sum of modules $e_iR/rR$, $r\in e_iRe_j$.

Suppose that $F$ is a field and $G$ is a finite group. We will be interested in when the group ring $FG$ is
serial. As we already mentioned, by Maschke's theorem, we may (and will) assume that the
\emph{characteristic $p$ of $F$ is finite and divides the order of $G$}.

An important necessary condition for seriality follows from Higman's theorem \cite{Hig} describing group rings
of finite representation type.

\begin{fact}\label{hig}
Let $F$ be a  field of characteristic $p$ dividing the order of $G$. If the group ring $FG$ is serial, then each
Sylow $p$-subgroup of $G$ is cyclic.
\end{fact}

As we have mentioned above, the converse is not true. For instance, let $F$ be an algebraically closed field and
let $G= \SL_2(p)$ for a prime $p\geq 5$. Then each Sylow $p$-subgroup of $G$ consists of $p$ elements, hence cyclic.
Further, it follows from \cite[p. 15]{Alp} that $FG$ has exactly $p$ simple modules $S_1, \dots, S_p$. Also,
for each $1< i< p-1$, the projective cover $P_i$ of $S_i$ is not uniserial: $\Jac(P_i)/\Soc(P_i)$ is the direct
sum of two simple modules $S_{p+1-i}$ and $S_{p-1-i}$.

In fact, this conclusion holds for any field $F$ of characteristic $p$, due to the following result by Eisenbud
and Griffith \cite{E-G} (see \cite{VKP} for another proof).

\begin{fact}\label{indep}
Let $F, F'$ be fields of characteristic $p$ dividing the order of $G$. Then the ring $FG$ is serial if and only
if the ring $F'G$ is serial.
\end{fact}

From this we derive a useful sufficient condition for seriality. Recall that a group $G$ is said to be
\emph{$p$-nilpotent}, if its Sylow $p$-subgroup admits a normal complement $H$, hence $G$ is a semidirect
product $H:P$. The following result is contained in \cite[Thm. 4.3]{K-P13a}.

\begin{fact}\label{p-nilp}
Let $F$ be a field of characteristic $p$ dividing the order of $G$. Further assume that $G$ is $p$-nilpotent with
a cyclic Sylow $p$-subgroup. Then $FG$ is a (left and right) principal ideal ring, in particular it is serial.
\end{fact}

For instance, it is well known that each group $G$ with a cyclic $2$-Sylow subgroup is $2$-nilpotent, hence,
for $p=2$, the group ring $FG$ is serial iff each Sylow 2-subgroup of $G$ is cyclic. Thus from now on
we will assume that $p> 2$ when investigating seriality.

This result can be extended to a larger class of groups. Recall that a group $G$ is said to be \emph{$p$-solvable},
if it admits a composition series whose consecutive factors are either $p$-groups, or $p'$-groups. If $G$ has a
cyclic Sylow $p$-subgroup (the case of our main interest), then, by an old results by Wielandt \cite{Wie}, we
conclude that $G$ is $p$-solvable iff it possesses a 4-term series $\{e\}\subset O_{p'}\subset K\subset G$
of completely characteristic subgroups, where $K$ is a semidirect product $O_{p'}:P$ (hence $p$-nilpotent) and
$G/K$ is a cyclic $p'$-group.

\begin{fact}\label{solv-ser}
Let $F$ be a field of characteristic $p$ dividing the order of $G$. Further assume that $G$ is a
$p$-solvable group with a cyclic Sylow $p$-subgroup. Then the group ring $FG$ is serial.
\end{fact}
\begin{proof}
If $F$ is algebraically closed, it was proved by Morita \cite{Mor} (and later by Srinivasan \cite{Sri}). It
remains to apply Fact \ref{indep}.
\end{proof}

Furthermore, we conclude from Morita that, for a $p$-solvable group $G$, the Jacobson radical of $FG$ is principal
as a left and right ideal, and the multiplicity of principal projective modules in a given block of $FG$ is
constant. It follows from \cite[Thm. 2.3]{K-P13b} that both results hold over any field of characteristic $p$.

Note that each group ring $FG$ of a finite group is \emph{quasi-Frobenius}, i.e.\ each projective module is
injective and vice-versa, for instance $FG$ admits a self-duality. From this it easily follows that, when
proving seriality, it suffices to verify that this ring is \emph{right serial}, i.e.\ that each principal
projective module $e_iFG$ is uniserial.

Furthermore, by Fact \ref{indep}, when checking that the group ring of a given group is serial (say, by
calculating radical series in MAGMA \cite{magma}), we may assume that $F$ is a prime field. When enlarging $F$,
the idempotents may split, hence the block structure of $FG$ will change, but the result is still serial. Here
is a typical example.

Let $2.S_4^-$ denote the \emph{double covering} of $S_4$. This group is $3$-solvable with a cyclic Sylow 
$3$-subgroup, hence the group ring $FG$ is serial for any field $F$ of characteristic $3$. By calculating 
in GAP \cite{gap}, one recovers (see \cite{Kukh}) the above Kupisch structure of this group ring.

\begin{exam}\label{2-cov}
Let $G= 2.S_4^-$ be the double covering of $S_4$.

1) If $F$ is the prime field $\F_3$, then $FG= M_3(F)^2\oplus B\oplus M_2(W)$, where $B$ is the serial block

$$
\bsm F[x]& F[x]\\ xF[x]&F[x]\esm \Bigl/\bsm x^2F[x]& xF[x]\\ x^2F[x]&x^2F[x]\esm
$$

\noindent and $W$ is the factor $\F_9[y,\al]/(y^3)$ of the skew polynomial ring with the Frobenius automorphism
$\lam\mapsto \lam^3$.

2) If $F$ is the Galois field $\F_9$, then $FG= M_3(F)^2\oplus B\oplus M_2(B)$.
\end{exam}

\section{Brauer trees}

In this section we will recall the main tool to investigate seriality of group rings - the Brauer trees.
This graph is defined for each block $B$ with a cyclic defect group $D$, and is a tree whose one (exceptional)
vertex is distinguished. For a definition and main properties of this object the reader is referred to the
classical \cite{Feit-b}, or contemporary \cite{L-P}; we will mention just few instances.

To calculate the Brauer tree one should fix a $p$-modular system $(K, R, F, \eta)$
(see \cite[Def. 4.1.18]{L-P}), where $R$ is a complete discrete valuation domain with the quotient field $K$ of
characteristic zero, and $\eta: R\to F$ maps $R$ onto its residue field. We will consider this system to be
fixed and large enough to split in $K$ some rational polynomials (called \emph{splitting system} in \cite{L-P}).
The nature of this system will be of no importance for us.

Suppose that $G$ is a group with a cyclic Sylow $p$-subgroup $P$. Because the defect group $D$ of each block
$B$ is a subgroup of $P$, it is cyclic, hence each block obtains its Brauer tree. For instance $P$ itself is
the defect group of the \emph{principal block} $B_0$, i.e. of the block which contains the trivial character.

The number of edges of the Brauer tree, and the multiplicity of the exceptional vertex admit a (more or less)
straightforward calculation.

\begin{fact}\label{brauer}
Let $G$ be a finite group with a (nontrivial) cyclic Sylow $p$-subgroup $P$.

1) The number of edges $e_0$ of the principal block of $G$ equals the index $|N_G(P)/C_G(P)|$ of the
centralizer of $P$ in $G$ in its normalizer. Further, $e_0$ divides $p-1$, and the multiplicity of the
exceptional vertex (if any) equals $m_0= (|P|-1)/e_0$.

2) If $B$ is an arbitrary block with a defect group $D$, then the number of edges $e$ in its Brauer tree divides
$e_0$. Furthermore the multiplicity of the exceptional vertex equals $(|D|-1)/e$.
\end{fact}
\begin{proof}
The first part follows from \cite[p. 173]{Blau} and \cite[Thm. 6.5.5]{Ben}. For the second claim one can argue
as follows. Let $Q$ be the subgroup of order $p$ in $D$. Then (see \cite[p. 212]{Ben}) $e$ divides
$|N_G(Q)/C_G(Q)|$. Now, by \cite[Prop. 1]{Saw}, the latter divides $|N_G(P)/C_G(P)|= e_0$.
\end{proof}

Note that $N_G(P)$ acts on $P$ by conjugation, and $e_0$ is the order of this action. Thus, if $H$ is a normal
subgroup of $G$ containing $P$, then $e_0(H)\leq e_0(G)$.

We say that a Brauer tree is a \emph{star}, if at most one vertex of this graph (its center) has a valency
exceeding $1$.  Here is the diagram of the star with 7 vertices, whose exceptional vertex is at its center.

$$
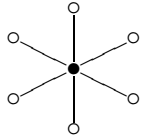
\vcenter{%
\def\labelstyle{\displaystyle}
\xymatrix@C=14pt@R=3pt{%
&*+={\circ}\ar@{-}[dd]&\\
*+={\circ}\ar@{-}[rd]&&*+={\circ}\ar@{-}[ld]\\
&*+={\bullet}\ar@{-}[dd]\ar@{-}[ld]\ar@{-}[rd]&\\
*+={\circ}&&*+={\circ}\\
&*+={\circ}&
}}
$$

\vspace{3mm}

Note that a \emph{line} (or \emph{open polygon}) with $e$ edges is a star iff $e=1$, or $e=2$ with the
exceptional vertex in the center. For instance the following line  with 3 edges is not a star.

$$
\vcenter{%
\def\labelstyle{\displaystyle}
\xymatrix@C=24pt@R=4pt{%
*+={\circ}\ar@{-}[r]&*+={\circ}\ar@{-}[r]&*+={\circ}\ar@{-}[r]&*+={\circ}
}}
$$

\vspace{3mm}

An irreducible character $\chi$ is said to be \emph{real}, if it is either real-valued, or $\chi$ is an
exceptional character taking real values on $p$-regular classes. By \cite[p. 3]{H-L} real characters of a
given block form a \emph{real stem} of its Brauer tree, which is a line.

For instance, let $G$ be the special unitary group $\SU_3(4^2)$. If $p=13$, then each Sylow $p$-subgroup is
cyclic. Further, from the character table we find that the Brauer tree of the principal block of $G$ is the
following line, where all non-exceptional characters are real valued. But each exceptional character
$\chi_{19\text{-22}}$ takes non-real values on each class of elements of order $13$.

$$
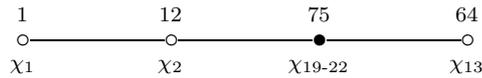
\vcenter{%
\xymatrix@C=50pt@R=10pt{%
*+={\circ}\ar@{-}[r]\ar@{}+<0pt,-10pt>*{_{\chi_1}}\ar@{}+<0pt,10pt>*{_1}&
*+={\circ}\ar@{-}[r]\ar@{}+<0pt,-10pt>*{_{\chi_2}}\ar@{}+<0pt,10pt>*{_{12}}
&*+={\bullet}\ar@{-}[r]\ar@{}+<0pt,-10pt>*{_{\chi_{19\text{-}22}}}\ar@{}+<0pt,10pt>*{_{75}}&
*+={\circ}\ar@{}+<0pt,-10pt>*{_{\chi_{13}}}\ar@{}+<0pt,10pt>*{_{64}}&
}}
$$

\vspace{3mm}

Thus the following is an instrumental criterion of seriality.

\begin{fact}(see \cite[Cor. 7.2.2]{Feit-b})\label{star}
Let $G$ be a finite group with a nontrivial cyclic Sylow $p$-subgroup and let $F$ be a field of characteristic
$p$. Then the ring $FG$ is serial if and only if each block of $G$ is a star whose exceptional vertex (if any)
is located at its center.
\end{fact}

For instance, let $G= A_5$ and $p=5$, hence each Sylow $5$-subgroup of $G$ is cyclic. Then the Brauer tree of
its principal block is the following:

$$
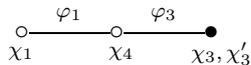
\vcenter{%
\xymatrix@C=30pt@R=4pt{%
*+={\circ}\ar@{}+<0pt,-8pt>*{_{\chi_1}}\ar@{-}[r]^{\phi_1}&
*+={\circ}\ar@{}+<2pt,-8pt>*{_{\chi_4}}\ar@{-}[r]^{\phi_3}&
*+={\bullet}\ar@{}+<4pt,-8pt>*{_{\chi_3,\, \chi'_3}}
}}
$$

\vspace{3mm}

Here the characters are indexed by degrees, for instance $\chi_1$ denotes the trivial character. Further,
the characters $\chi_3, \chi'_3$ form the exceptional vertex of multiplicity $2$. Because this vertex is
not in the center, the group ring of $A_5$ over any field of characteristic 5 is not serial.

The same results can be achieved (but with more efforts) by splitting idempotents in $FG$ for $F= \F_5$. Namely
(see \cite[p. 259]{B-O}) the group ring $F A_5$ decomposes as $M_5(F)\oplus C$, where

$$
C= \bsm Q&Q&Q&X\\Q&Q&Q&X\\Q&Q&Q&X\\Y&Y&Y&T\esm\,
$$

\noindent is a block of dimension $35$, which is the blow-up of its $(3,4)$-minor $\bsm Q&X\\Y&T\esm$,
where the diagonal ring $Q$ has dimension $3$, $T$ is 2-dimensional, and bimodules $X, Y$ are 1-dimensional.
Further, (see \cite{K-P13b}) the latter ring is isomorphic to the following factor-ring:

$$
R= \bsm F[x]& xF[x]\\ xF[x]&F[x]\esm \Bigl/\bsm x^3F[x]& x^2F[x]\\ x^2F[x]&x^2F[x]\esm\,.
$$

Taking $r= \ov x\in e_1 Re_1$ and $s= \ov x\in e_1Re_2$, it is easily seen that the module $e_1R$ is not
uniserial.

Clearly the above criterion (Fact \ref{star}) is equivalent to the following: each irreducible $p$-modular
character lifts \emph{uniquely} to an irreducible ordinary character via the $p$-modular system. For instance,
for $A_5$ and $p=5$, each $p$-modular character can be lifted, but the uniqueness fails: the restrictions of
ordinary characters $\chi_3, \chi'_3$ to $5'$-classes is the Brauer character $\phi_3$. This uniqueness
requirement disappeared from the characterization of serial group rings in Janusz \cite[Cor. 7.5]{Jan}.

Note that seriality is a Morita invariant property. Furthermore, each factor ring of a serial ring is
serial. For instance, if $H$ is a normal subgroup of a group $G$, then the seriality of $FG$ implies that
the ring $F (G/H)$ is serial.

We say that a block $B$ of $FG$ \emph{covers} a block $b$ of $FH$, if there are irreducible characters
$\chi\in B$ and $\xi\in b$ such that $\xi$ is a component of $\chi$ restricted to $H$. For instance, the
principal block of $G$ covers the principal block of $H$. We will put to use the following fact.

\begin{fact}(see \cite[Cor. 6.2.8]{Feit-b})\label{feit}
Let $F$ be a field of characteristic $p$ dividing the order of $G$ and let $H$ be a normal subgroup of $G$.

1) Let $B$ be a block of $FG$ which covers a block $b$ of $FH$. Assume that $H$ contains a defect group of $B$.
Then $B$ is a serial ring if and only if $b$ is a serial ring.

2) For instance, if the index of $H$ in $G$ is coprime to $p$, then the ring $FG$ is serial if and only if the
same holds true for $FH$.
\end{fact}

Note (see \cite[L. 4.4.12]{Feit-b}) that the annihilator of the principal block of any group $G$ is its
largest normal $p'$-subgroup $O_{p'}(G)$. Thus, when investigating seriality of the principal block, we may
factor out this subgroup.

We will also need the following result on groups with cyclic Sylow $p$-subgroups, which deserves a better
publicity.

\begin{fact}(see \cite[L. 5.2]{Blau} and \cite[L. 6.1]{Naer})
Let $G$ be a non-$p$-solvable group with a nontrivial cyclic Sylow $p$-subgroup $P$. Then there exists the
least normal subgroup $K$ of $G$ properly containing $O_{p'}(G)$. Furthermore, $K$ contains $P$ and the factor
group $H= K/O_{p'}(G)$ is simple non-abelian.
\end{fact}

Note that this result does not extend to $p$-solvable groups. Namely, let $G$ be the dihedral group $D_{18}$ and
$p=3$. Then $O_{p'}(G)=1$ and $K= C_3$ does not contain $P= C_9$.

Thus, each non-$p$-solvable group $G$ with a cyclic Sylow $p$-subgroup admits the following 4-term normal
series, similar to $p$-solvable groups: $0\subset O_{p'}\subset K\subset G$. Further, the index of $K$ in
$G$ is coprime to $p$, hence, by Fact \ref{feit}, when investigating seriality, we may assume that $K= G$.
Also, the group ring of the simple group $H= K/O_{p'}$ must be serial and, if so, then the  principal block of
$FG$ is serial.

\section{Simple groups with serial group rings}\label{S-simple}

In this section we will start proving Theorem \ref{t-main}, i.e. classifying finite simple groups $H$ with
serial group rings. Because each abelian group is $p$-solvable, its group ring is serial, hence assume
that $H$ is non-abelian.

Most cases has been already considered. We will just briefly comment on some instances.

1) The case of alternating groups was considered in \cite{K-P14} using Scope's theorem \cite{Sco} on Morita
equivalence of blocks of symmetric groups.

2) The analysis of the $\PSL$-series was completed in \cite{Kukh} using Burkhard's paper \cite{Burk76}.

3) The case of sporadic simple groups was settled in \cite{K-P15a} by browsing known Brauer trees from
Hiss--Lux \cite{H-L}. In fact, few facts from Blau (like the parity of centralizers - see \cite[Cor. 1]{Blau})
simplifies matters greatly, and the cross-naught method of labeling Brauer trees (see \cite[Def. 2.1.12]{H-L})
was used for large groups, when the Brauer trees are not available.

As a kind of peculiarity note that the principal block of the Mathieu group $M_{23}$ for $p=5$ is serial,
because (see \cite[p. 104]{H-L}) its Brauer tree is a star with 4 edges without an exceptional vertex. However,
there is a nonprincipal block whose Brauer tree is a line with 2 edges, whose exceptional vertex has multiplicity
2 and is attached to the end, hence this block is not serial.

4) The analysis of Suzuki groups in \cite{K-P15a} is based on Burkhardt \cite{Burk79}. For instance, when
$5\mid q+r+1$, then each Sylow $5$-subgroup is cyclic, and the Brauer tree of the principal block has the
following shape:

$$
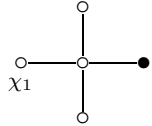
\vcenter{%
\xymatrix@C=16pt@R=14pt{%
&*+={\circ}\ar@{-}[d]&\\
*+={\circ}\ar@{-}[r]\ar@{}+<0pt,-8pt>*{_{\chi_1}}&*+={\circ}\ar@{-}[d]\ar@{-}[r]&*+={\bullet}\\
&*+={\circ}&
}}
$$

\vspace{1mm}

This block is serial iff the multiplicity of the exceptional vertex is 1, hence if $25$ does not divide
$q+r+1$. For instance, this is the case, when $n=3$, i.e. $q= 128$.

A similar situation occurs (see \cite{K-P17}), when $G$ is the Ree group ${}^2G_2(q^2)$. Namely, if $p$ divides
$q^2+ \sq 3\, q+ 1$, then, according to \cite[Thm. 4.3]{Hiss-91}, each Sylow $p$-subgroup is cyclic and
the Brauer tree of the principal block of $G$ is a star with 6 edges, whose exceptional vertex is on the
boundary. Because the multiplicity equals $(|P|-1)/6$, if the ring $FG$ is serial, we conclude that $p=7$
and $|P|=7$. For instance, this is the case when $q^2= 3^{13}$.

The remaining cases of exceptional Lie groups are ruled out in \cite{K-P17} using known information on
Brauer trees, for instance see \cite{H-L98} and \cite{HLM}.

5) The symplectic, unitary and orthogonal groups defined over fields with odd number of elements were
investigated in \cite{K-P16}. The main tool was an old result by Fong and Srinivasan \cite{F-S}. It says
that in this case (and when $p\neq 2$ does not divide $q$) the Brauer trees of all cyclic blocks of certain
(not necessarily simple) groups in classical series are lines.

Thus, to complete the proof of Theorem \ref{t-main}, it remains to consider the case when $G$ is a symplectic,
orthogonal or unitary simple group defined over a field with even number of elements. The proof proceeds
similarly to \cite{K-P16}, but first we will need a substitute for Fong and Srinivasan result.

\begin{fact}\label{polyg}
Let $G$ be one of groups $\Sp_{2m}(q)$, $\GO^{\pm}_{2m}(q)$, where $q$ is even. Then the Brauer tree of each
cyclic block of $G$ is a line.
\end{fact}
\begin{proof}
By \cite[Thm.]{Gow}, each element of $G$ is a product of two involutions. From this it follows easily that
each element of $G$ is conjugated to its inverse, hence each character of $G$ is real-valued. Thus the Brauer
tree of each cyclic block coincides with its real stem, hence is a line.
\end{proof}

We copy from Stather \cite[p. 549]{Stat} the table (see Table \ref{t-1}) describing the sizes of Sylow
$p$-subgroups of certain classical groups, when $p>2$ does not divide $q$. Here $|H|_p$ denotes the order of
a Sylow $p$-subgroup of a group $H$, and $\lf k/l\rf$ is the integer part of the fraction $k/l$.

\begin{table}[b]
\caption{Orders of Sylow $p$-subgroups of classical groups}\label{t-1}
\begin{tabular}{|p{1.6cm}|p{4.0cm}|p{2.8cm}|p{1.0cm}|}
\hline
Group & \hspace{0mm} Condition on $d$& Order of a Sylow $p$-subgroup & Sylow type\\
\hline
\multirow{2}{*}{$\Sp_{2m}(q)$} & $d$ even & $|\GL_{2m}(q)|_p$ & A\\
\cline{2-4}
&$d$ odd & $|\GL_m(q)|_p$ & B\\
\hline
\multirow{3}{*}{$\GO^+_{2m}(q)$} & $d$ odd & $|\GL_m(q)|_p$ & B\\
\cline{2-4}
&$d$ even and $\lf d/2m\rf$ odd & $|\GL_{2m-2}(q)|_p$ & A\\
\cline{2-4}
& otherwise & $|\GL_{2m}(q)|_p$ & A\\
\hline
\multirow{3}{*}{$\GO^-_{2m}(q)$} & $d$ odd & $|\GL_{m-1}(q)|_p$ & B\\
\cline{2-4}
&$d$ even and $\lf d/2m\rf$ even & $|\GL_{2m-2}(q)|_p$ & A\\
\cline{2-4}
& otherwise & $|\GL_{2m}(q)|_p$ & A\\
\hline
\multirow{2}{*}{$\GU_n(q^2)$} & $d \equiv 2\pmod 4$& $|\GL_n(q^2)|_p$ & B\\
\cline{2-4}
& otherwise & $|\GL_{\lf n/2\rf}(q^2)|_p$ & A\\
\hline
\end{tabular}
\end{table}

We will consider each case in turn. Since none of simple classical groups is 2-nilpotent, we may assume that
$p> 2$. Furthermore (say, by \cite[Prop. 5.1]{S-W}), in the defining characteristic each Sylow $p$-subgroup
is not cyclic, hence we will also assume that \emph{$p$ does not divide $q$}.

Recall also a relevant information on Sylow subgroups of general linear groups. Suppose that $d$ is the order
of $q$ modulo $p$. Then $G= \GL_d(q)$ contains a copy of the multiplicative group $\F_{q^d}^*$ of the Galois
field, the so-called \emph{Singer cycle}, and a Sylow $p$-subgroup of $G$ can be chosen within this cycle.
Then the centralizer of $P$ in $G$ coincides with the cycle, and the normalizer of $P$ is generated over the
centralizer by an element $y$, which acts by conjugation on a generator $\al$ of $P$ as $\al^y= \al^q$,
in particular the index $|N_G(P)/C_G(P)|$ equals $d$.

Note also that a Sylow $p$-subgroup of $\GL_m(q)$ is cyclic and nontrivial iff $m/2< d\leq m$. In this case
$P$ can be chosen within the copy of $\GL_d(q)$ embedded in the left upper corner.

The idea of forthcoming proofs is the following. If the ring $FG$ is serial, then $d$ cannot be too small
comparing with the size of matrices, otherwise $P$ is not cyclic. Further, $d$ cannot be too large, because
this would imply $e_0\geq 3$. This effectively restricts the size of matrices for which serial rings occur,
and the remaining cases are just few.

\section{Symplectic groups}\label{symp}

Because $q$ is even, one may define symplectic groups using a non-singular symmetric bilinear form $f$ on a
$2m$-dimensional vector space $V$ over the Galois field $\F_q$. This form is unique up to the choice of a
basis in $V$, and is given by a matrix $W$ such that $W= W^t$, where $t$ means transpose.

A \emph{symplectic group}, $\Sp_{2m}(q)$, consists of invertible matrices $A$ of order $2m$ which preserve $f$,
i.e. $A W A^t= W$. This group has the following order:

$$
|\Sp_{2m}(q)|= q^{m^2}\cdot (q^2-1)\cdot (q^4-1)\cdot\ldots \cdot (q^{2m}-1)\,.
$$

For instance, one could take \rule{0pt}{6mm}$W= \bsm 0&I_m\\I_m&0\esm$. Then the rule
$A\mapsto \bsm A&0\\ 0& A^{-t}\esm$ defines an embedding from $\GL_m(q)$ into $\Sp_{2m}(q)$.

Since $q$ is even, $G= \Sp_{2m}(q)$ coincides with its projective variant, and is simple with few exceptions:
$\PSp_2(2)\cong S_3$ and $\PSp_2(4)\cong S_6$ (see Atlas \cite[p.~x]{atlas}). Further, $\PSp_2(q)$ is isomorphic to
$\PSL_2(q)$. This case has been already analyzed in \cite{Kukh}, hence seriality occurs just in cases 2) and
3) of Theorem \ref{t-main}. Thus we may assume that $m> 1$, and we prove that no serial group rings occur
in the case.

\begin{prop}\label{symp}
Let $F$ be a field of characteristic $p$ dividing the order of a (simple) symplectic group $G= \Sp_{2m}(q)$,
where $m> 1$ and $q$ is even. Then the group ring $FG$ is not serial.
\end{prop}
\begin{proof}
Let $d$ denote the order of $q$ modulo $p$. First we consider the case when $d=1$, i.e. $p$ divides $q-1$.
Because $m\geq 2$, each Sylow $p$-subgroup of $\GL_m(q)$ is not cyclic. The above mentioned diagonal
embedding shows that the same holds true for $\Sp_{2m}(q)$, contradicting seriality.

Otherwise $d\geq 2$ and we will show that the index $e= |N_G(P)/C_G(P)|$ is at least $d$. Because $e$ is the
number of edges in the Brauer tree of the principal block of $G$, which is a line, non-seriality would follow.

The order and structure of $P$ depends on the parity of $d$ - see Table 1.

\textbf{Case A: $d$ is even}. In this case $P$ coincides with a Sylow $p$-subgroup of the ambient group
$\GL_{2m}(q)$. Because $P$ is nontrivial and cyclic, we conclude that $m< d\leq 2m$. Further, $m\geq 2$
yields $d> 2$, hence $d\geq 4$.

Under a suitable choice of the matrix $W$, the group $\Sp_d(q)\times \Sp_{2m-d}(q)$ is embedded into
$\Sp_{2m}(q)$ as a blocked diagonal, hence $P$ can be chosen in the left upper corner $\GL_d(q)$. Thus
it suffices to show that the index $e_d= |N_{G_d}(P)/C_{G_d}(P)|$ is at least $d$.

Let $\al$ be a generator of a Sylow $p$-subgroup $P_d$ of $\Sp_d(q)$. From \cite[L. 4.6]{Stat}  it follows
that the factor $N_{G_d}(P_d)/C_{G_d} (P_d)$ is generated by an element $y$ acting by conjugation as
$\al^y= \al^d$. Because this action has order $d$, we conclude that $e_d\geq d\geq 4$, as desired.

\textbf{Case B: $d$ is odd}. From Table \ref{t-1} we derive that in this case the order of $P$ is the same
as the order of a Sylow $p$-subgroup $P'$ of $\GL_m(q)$. We may assume that $P$ is the image of $P'$, when
$\GL_m(q)$ is embedded into $\Sp_{2m}(q)$ diagonally (as above). Because $P'$ is cyclic, hence $m/2< d\leq m$.
Since $m\geq 2$ and $d$ is odd, we derive $d\geq 3$, hence it suffices to check that $e\geq d$.

As above we may assume that $m=d$. Consider an element $y'\in \GL_d(q)$ which acts by conjugation on $P'$
as an automorphism of order $d$. Then the diagonal image of this element belongs to the normalizer of $P$ and
acts as an automorphism of order $d$ on this subgroup.
\end{proof}

\section{Orthogonal groups}\label{S-orthog}

If $q$ is even, then the odd-dimensional orthogonal group $\O_{2m+1}(q)$ is isomorphic to the symplectic group
$\PSp_{2m}(q)$, and this case has been already considered. In the even-dimensional case we use
\cite{DLLB} as a guide for definitions.

\textbf{Case $+$.} A \emph{general orthogonal group of $+$ type}, $\GO_{2m}^+(q)$, consists of invertible
matrices $A$ of order $2m$ which preserve the nonsingular quadratic form given by the matrix
$W= \bsm 0&I_m\\ 0&0\esm$, i.e. $AWA^t= W$. One can define a morphism from this group to $\{\pm 1\}$ by
sending an element to $(-1)^k$,  where $k$ is the dimension of its fixed subspace when acting on $V$. The
kernel of this map is the \emph{orthogonal group}, $\O_{2m}^+(q)$, which is of index 2 in $\GO^+_{2m}(q)$.
The order of this group is the following.

$$
|\O_{2m}^+(q)|= q^{m(m-1)}\cdot (q^m-1)\cdot \prod_{i=1}^{m-1} (q^{2i}-1)\,.
$$

For $m\leq 3$ these groups are isomorphic to groups from other series, namely
$\O_4^+(q)\cong \PSL_2(q)\times \PSL_2(q)$ is not simple, and $\O_6^+(q)\cong \PSL_4(q)$, hence only the case
$m> 3$ is of potential interest.

\begin{prop}\label{plus}
Let $F$ be a field of characteristic $p$ dividing the order of the orthogonal group $G=\O_{2m}^+(q)$, where
$q$ is even and $m> 3$. Then the group ring $FG$ is not serial.
\end{prop}
\begin{proof}
As usual we may assume that $p> 2$ does not divide $q$, and that a Sylow $p$-subgroup $P$ of $G$ is nontrivial
and cyclic.  Since the index of $G$ in $H= \GO^+_{2m}(q)$ equals 2, hence $P$ is a Sylow $p$-subgroup of $H$.

If $p$ divides $q-1$, then $P$ is not cyclic, hence suppose that this is not the case. Thus, if $d$
denotes the order of $q$ modulo $p$, then $2\leq d\leq 2m-2$.

Recall (see Fact \ref{polyg}) that the Brauer tree of the principal block $B'_0$ of $H= \GO^+_{2m}(q)$ is a
line with $e'$ edges. Further, Fact \ref{feit} implies that the principal block of $G$ is serial iff the same
holds true for $H$, hence it suffices to prove that $e'> 2$.

We consider the possibilities for $P$ given in Table \ref{t-1}.

\textbf{Case A: $d$ is odd}. Then the size of $P$ equals $|\GL_m(q)|_p$. Since $P$ is nontrivial and cyclic, we
obtain $m/2< d\leq m$, hence $m\geq 4$ yields $d> 2$. Note that the image of $\GL_m(q)$ via the diagonal
embedding $A\mapsto \bsm A&0\\0&A^{\,-t}\esm$ is contained in $\GO^+_{2m}(q)$. Now, by considering the normalizer
of a Sylow $p$-subgroup in $\GL_d(q)$ and this embedding, we derive $e'\geq d> 2$, as desired.

\textbf{Case B: $d$ is even}. Then $d\leq 2m-2$ implies that the integer part of the fraction $d/2m$ is zero,
hence $P$ is a Sylow $p$-subgroup of $\GL_{2m}(q)$. Since $P$ is cyclic and nontrivial, we derive that $m< d$. 
Now, from \cite[L. 4.6]{Stat} it follows that $e'\geq d$. Then $d> 4$ yields $e'> 4$, as desired.
\end{proof}

\textbf{Case $-$.} Let $\ga$ be a primitive element of $\F_{q^2}$ and set $a= \ga+ \ga^q$, $b= \ga^{q+1}$.

A \emph{general orthogonal group of $-$ type}, $\GO_{2m}^-(q)$, consists of invertible matrices $A$ of order
$2m$ which preserve the nonsingular quadratic form given by the matrix
\rule{0pt}{7mm}$W= \bsm 0&I_{m-1}&0&0\\0&0&0&0\\0&0&1&a\\0&0&0&b\esm$,
i.e. $AWA^t= W$; and the \emph{orthogonal group}, $\O_{2m}^-(q)$, is defined as for the $+$ case.
The\rule{0pt}{4mm} order of this group is the following.

$$
|\O_{2m}^-(q)|= q^{m(m-1)}\cdot (q^m+1)\cdot \prod_{i=1}^{m-1} (q^{2i}-1)\,.
$$

For $m\leq 3$ these groups are isomorphic to groups from other series, namely $\O_4^-(q)\cong \PSL_2(q^2)$
(which we already know) and $\O_6^+(q)\cong \PSU_4(q)$ (which we consider later). Thus only the case $m> 3$
is of potential interest.

\begin{prop}\label{minus}
Let $F$ be a field of characteristic $p$ dividing the order of the orthogonal group $G=\O_{2m}^-(q)$, where
$q$ is even and $m> 3$. Then the group ring $FG$ is not serial.
\end{prop}
\begin{proof}
We use the same notations and assumptions as in the beginning of the proof of Proposition \ref{plus}. The only
difference is in the weaker conclusion $d\leq 2m$, because $d=2m$ may occur, when $p$ divides $q^m+1$. Again,
it suffices to show that $e'> 2$. From Table \ref{t-1} we obtain the following possibilities for $P$.

\textbf{Case B: $d$ is odd}. Then $P$ is the image, via the above skew diagonal embedding, of a Sylow 
$p$-subgroup $P'$ of $\GL_{m-1}(q)$. Since $P'$ is cyclic, from $(m-1)/2< d$, we conclude that $d\geq 3$. 
Applying \cite[L. 4.6]{Stat} we obtain $e'\geq d\geq 3$, as desired.

\textbf{Case A: $d$ is even}. This case splits in two subcases.

If the integer part of the fraction $d/2m$ is odd, then $d\leq 2m$ implies $d= 2m$. From Table \ref{t-1} we
see that $P$ is a Sylow $p$-subgroup of $\GL_{2m}(q)$. Since $P$ is cyclic, from $m< d$ we derive $d\geq 6$. 
Again \cite[L. 4.6]{Stat} gives $e'\geq d\geq 6$, as desired.

Otherwise the integer part of $d/2m$ is zero, hence $P'$ is a Sylow $p$-subgroup of $\GL_{2m-2}(q)$. Since
$P'$ is cyclic, therefore $m-1< d$ implies $d\geq 4$. One more application of \cite[L. 4.6]{Stat} yields
$e'\geq d\geq 4$.
\end{proof}

\section{Unitary groups}\label{S-unit}

It is a common knowledge that Fong--Srinivasan theorem also holds for $G= \GU_n(q)$ when $q$ is even, i.e. the
Brauer tree of each cyclic block of $G$ is a line. It would cut few lines from proofs, but is very difficult to 
find a proper reference. We will be content with the following.

\begin{lemma}\label{rob}
Let $G= \SU_n(q)$ be a quasi-simple special unitary group with a nontrivial cyclic Sylow $p$-subgroup. Then
the Brauer tree of the principal $p$-block of $G$ is a line.
\end{lemma}
\begin{proof}
It follows from \cite[Sect. 6]{Rob} that each non-exceptional character in the principal block of $G$ is
rational-valued. By looking at the real stem we conclude the desired.
\end{proof}

Now it follows from \cite{Feit} that the Brauer trees of the principal blocks of $\SU_n(q^2)$ and $\GU_n(q^2)$
are obtained by unfolding the same line. Thus to prove that both trees are lines with the same number of
edges it sufficient to check that the indices $|N(P)/C(P)|$ in both groups are equal, but it is hard to find a
reference.

Let $\ov{\phantom{r}}$ denote the involution $\ov a= a^q$ of the Galois field $\F_{q^2}$ and let $V$ be an
$n$-dimensional vector space over this field. There exists (an essentially unique) non-singular
conjugate-symmetric sesquilinear form $f: V\times V\to \F_{q^2}$, i.e. $f(u,v)= \ov{f(v,u)}$ holds for all
$u, v\in V$. If $f$ is given by a matrix $W$, then $W= \ov W^t$.

A \emph{general unitary group}, $\GU_n(q^2)$, consists of matrices $A\in \GL_n(q^2)$ preserving $f$, i.e.
$A W\ov A^{\, t}= W$. In particular, if $W= I_n$, then $\GU_n(q^2)$ consists of matrices $A$ such that
$A\cdot \ov A^{\,t}= I_n$. The order of this group is the following.

$$
|\GU_n(q^2)|= q^{n(n-1)/2}\cdot (q+1)\cdot (q^2-1)\cdots\ldots\cdot (q^n- (-1)^n)\,.
$$

The unitary matrices of determinant 1 form a normal subgroup in $\GU_n(q^2)$ of index $q+1$, the
\emph{special unitary group}, $\SU_n(q^2)$. The center $Z$ of this group consists of scalar matrices, and
is of order $(n, q+1)$. The factor group $\PSU_n(q^2)= \SU_n(q^2)/Z$ is the
\emph{projective special unitary group}. If $n\geq 3$, then this group is simple, with $\PSU_3(2^2)$ being the
only exception.

Note that $\PSU_2(q^2)\cong \PSL_2(q)$ hence, when investigating seriality, we may assume that $n\geq 3$.
Thus it suffices to prove the following proposition.

\begin{prop}\label{unit}
Let $F$ be  a field of characteristic $p$ dividing the order of a simple unitary group $H=\PSU_n(q^2)$, where
$n\geq 3$ and $q$ is even. Then the group ring $FH$ is serial if and only if $n=3$ and $p>2$ divides $q-1$.
\end{prop}
\begin{proof}
First we consider the case $n=3$. If $p\mid q-1$, then it follows from Geck \cite[Thm. 6.1]{Geck} that the group
ring of $\SU_3(q^2)$ is serial. Since $H$ is a factor of this group, the group ring $FH$ is also serial.
Furthermore, it is also derived from Geck that no other serial rings occur for $\SU_3(q)$, because either $P$
is not cyclic, or the principal block is not serial. If $p$ does not divide $q+1$ or $3$, then both properties
are inherited by $G$. In the remaining case, when $H= \PSU_3(q^2)$ and $p=3$ divides $q+1$, it is easily
checked that each Sylow $3$-subgroup of $H$ is not cyclic.

Thus we may assume that $n\geq 4$. In this case, if $p$ divides $q\pm 1$, then each Sylow $p$-subgroup of $H$
is not cyclic. Otherwise, if $d$ is the order of $q$ modulo $p$, then $2< d\leq 2n$.

Further, the principal block of $H$ coincides with the principal block $B_0$ of $G= \SU_n(q^2)$, and the Brauer
tree of $B_0$ is a line. Thus it suffices to show that the number of edges $e$ in this line exceeds 2.
If $P$ denotes a Sylow $p$-subgroup of $G$ then it is a Sylow $p$-subgroup of $\GU_n(q^2)$. We consider
various possibilities for $P$ given in Table \ref{t-1}.

\textbf{Case B: $d\equiv 2 \pmod 4$.} In particular $d> 2$ yields $d\geq 6$.

In this case $P$ is a Sylow $p$-subgroup of the ambient group $\GL_n(q^2)$. Note that the order $f$ of $q^2$
modulo $p$ equals $d/2$. Because $P$ is cyclic, we conclude that $n/2< f\leq n$, i.e. $n< d\leq 2n$.

If $n=f$ then the centralizer $C_G(P)$ is contained in the Singer cycle, hence is of odd order. Because $G$ has
no involution in the annihilator of $B_0$ and at least two classes of involutions, as in Blau
\cite[proof of Thm. 1 and Cor. 1]{Blau}, we conclude that $B_0$ is not a star, hence $e\geq 3$ in this group.
By considering the diagonal embedding from $\SU_f(q^2)\times \SU_{n-f}(q^2)$ into $\SU_n(q^2)$, we obtain the same
conclusion for $\SU_n(q^2)$.

\textbf{Case A: $d\equiv 0, 1, 3 \pmod 4$.} Let $n= 2m+ \eps$, where $\eps= 0, 1$. From Table \ref{t-1} we see
that the order of $P$ equals the order of a Sylow $p$-subgroup $P'$ of $\GL_m(q^2)$. If $\eps=0$, then choosing
$W= \bsm 0&I_m\\I_m&0\esm$, we obtain the embedding from $\GL_m(q^2)$ into $\GU_n(q^2)$ which sends $A$ to
$\bsm A&0\\0& \ov{A}^{\,-t}\esm$, and a similar embedding takes place if $\eps=1$. Because $p$ does not divide 
$q^2-1$, the generator $\al'$ for $P'$ is in $\SL_n(q^2)$, hence its image belongs to $\SU_n(q^2)$.

Recall that $d$ is the order of $q$ modulo $p$, and $f$ is the order of $q^2$ modulo $p$. First we consider
the case when $d$ is \emph{odd}, hence $f= d> 2$.

Because $P'$ is cyclic and nontrivial, we conclude that $m/2< f\leq m$. As above, to estimate $e$, we may assume
that $n= 2d$. Choose an element in $\GL_d(q^2)$ which acts by conjugation on $P'$ as an automorphism of order
$f$. By multiplying by a constant we may assume that this element has determinant 1. Expanding diagonally, we
conclude that $e\geq f> 2$, as desired.

Thus it remains to consider the case when $d$ is \emph{divisible by 4}, in particular $d\geq 4$. Then $d= 2f$
yields $f\geq 2$. If $d> 4$, then, using the diagonal embedding, we conclude that $e\geq f> 2$.

Thus we may assume that $d=4$ and $f=2$, i.e. $G= \SU_4(q^2)$ and $p\mid q^2+1$. In this case the generator
$A=\al'$ for $P'$ can be chosen such that $A\cdot \ov A^{\, t}= I_2$. As above, using the diagonal embedding
we obtain that $e\geq 2$. But also the matrix $W$ normalizers $P$, hence $e\geq 4$, as desired.

In fact, $e=4$ in this case by looking at the generic character table for $\GU_4(q)$. Say, the character degrees 
can be obtained formally extending Carter \cite[p. 465]{Car} to the case of odd $q$.

$$
\vcenter{%
\def\labelstyle{\displaystyle}
\xymatrix@C=70pt@R=10pt{%
*+={\circ}\ar@{-}[r]\ar@{}+<0pt,10pt>*{_1}&
*+={\circ}\ar@{-}[r]\ar@{}+<0pt,10pt>*{_{q^3(q^2-q+1)}}&
*+={\bullet}\ar@{-}[r]\ar@{}+<5pt,10pt>*{_{(q-1)(q+1)^3(q^2-q+1)}}&
*+={\circ}\ar@{-}[r]\ar@{}+<0pt,10pt>*{_{q^6}}&
*+={\circ}\ar@{}+<0pt,10pt>*{_{q(q^2-q+1)}}\\
}}
$$

\vspace{3mm}

\end{proof}

\section{Small groups}\label{S-small}

In this section we will verify (modulo some calculations in MAGMA, which we omit) Conjecture \ref{c-main} for 
small groups.

\begin{prop}\label{small}
Conjecture 1 holds true for all groups $G$ of order $\leq 10^4$.
\end{prop}
\begin{proof}
Suppose otherwise. Thus, for some prime $p$, there exists a small non-$p$-solvable group $G$ with a cyclic
Sylow $p$-subgroup $P$ and the normal series $\{e\}\subset O_{p'}\subset K\subset G$ such that the group ring
of the simple group $H= K/O_{p'}$ over $\F_p$, but the ring $FG$ is not serial.

Suppose that $G$ is a minimal such counterexample. By Fact \ref{feit}, it follows that $K= G$ and $H$ is a
simple nonabelian finite group of order $\leq 10^4$, whose group ring is serial. Further, this fact implies that
$G$ contains no proper normal subgroup containing $P$. Also, if $p=3$ and $|P|=3$, then the multiplicity of
the exceptional vertex equals $(3-1)/2=1$. Thus in this case seriality of $H$ implies seriality of $G$,
a contradiction. Thus we may assume that, if $p=3$ divides the order $H$, then $|P|\geq 9$.

By these remarks, from the list of simple groups of size $\leq 5\cdot 10^3$ whose group rings are serial
(see Theorem \ref{t-main}) only the following are potential candidates for $H$ in the projected counterexample.

1) $H= \PSL_2(8)$ of order $504$, and $p=7$, $|P|= 7$.

This is the most difficult case: we have to check all extensions $G= U.H$, where $U= O_{p'}$ has order at most 
$19$, and we use the MAGMA command ExtensionsOfSolubleGroup for this search. For instance, for $|O_{p'}|=24$ 
there are $15$ extensions all of which contain $H$ as a normal subgroup.

2) $H= \PSL_2(11)$ of order $660$, and $p=5$, $|P|=5$.

In this case we have to investigate extensions $G= U.H$, where $U= O_{p'}$ has order at most $14$. For instance, 
the only extensions $C_2.H$ are the direct product $C_2\times H$; and $\SL_2(11)$, where the normalizer of $P$ 
coincides with $G$. However the $\SL$-series has been already considered: it follows from \cite{Kukh} that
the ring $FG$ is serial, a contradiction.

3) $H= \PSL_2(19)$ of order $3420$, and $p=3$, $|P|=9$.

Here the only case is $O_{p'}= C_2$. The only extension which is not a direct product is $\SL_2(19)$, for which
seriality is known.

4) $H= \PSL_2(16)$ of order $4080$, and $p=5$, $|P|=5$.

Here the only extension $C_2.H$ is the direct product $C_2\times H$.
\end{proof}

By pushing harder, one probably could improve the above estimate, but not beyond $60\cdot 504= 30,240$, 
because $|A_5|=60$ and $A_5$ is not solvable.

\section{Discussion}\label{S-dis}

There are few open question on serial group rings of finite groups we would like to address.

\begin{ques}\label{3}
Is Conjecture \ref{c-main} true for $p=3$?
\end{ques}

Note that in this case $e$ divides $2$, hence only the case when $e=2$ and $|P|\geq 9$ is of interest.
Further, under these restrictions, the list of simple groups in Theorem \ref{t-main} contains only groups
in $\PSL_2$, $\PSL_3$ and $\PSU_3$-series. Also, the principal block of $G$ is serial, hence we should
investigate non-principal blocks with large defect groups. If this block is non-serial, then $e=2$ and the
exceptional character occurs at the end of the line. Some information on values (hence degrees) of
characters in this block can be extracted from \cite[Thm. 7.2.16]{Feit-b}, but we were not able to draw
a decisive conclusion. The feeling is that, even in this case, the main conjecture is based on a little
empirical evidence.

The following simple question also shows our limits.

\begin{ques}\label{h}
Suppose that $F$ is a field of characteristic $p$ dividing the order of $G$ and let $H$ be a normal subgroup of
$G$. Is it true that, if $FG$ is serial, then $FH$ is serial?
\end{ques}

Of course,  modulo the main conjecture, this question has an affirmative answer.

Note that Blau \cite{Blau} considered a similar question: when the Brauer tree of the principal block of
a group $G$ is a star (with no restriction of the position of the exceptional character). This is the same
as each $p$-modular irreducible character in the principal block lifts to an ordinary character.

Comparing with seriality, there are some simplifications: from the very beginning one can factor out the
subgroup $O_{p'}$. The list of finite simple groups satisfying Blau's condition will be a bit larger than
the one given in Theorem \ref{t-main}. For instance, the Mattieu group $M_{23}$ will get there when $p=5$,
and there will be more varieties of Suzuki and Ree groups. However, the difficulty may lay when passing from a
normal subgroup of coprime index to $G$ and vice-versa.

By $\Z_{(p)}$ we will denote the localization of integers with respect to a prime ideal $p\,\Z$.

\begin{ques}\label{zp}
Describe finite groups $G$ such that the group ring $\Z_{(p)}G$ is serial.
\end{ques}

Of course, seriality of this ring implies seriality of the ring $\F_p G$. Further, because the valuation
domain $\Z_{(p)}$ is not complete, the indecomposable idempotents lift rarely, hence the list of simple groups in
Theorem \ref{t-main} will get scarce. However we will face the other difficulties, starting from the lack
of Maschke's theorem.


\end{document}